\documentclass[12pt]{amsart}
\usepackage{amssymb}
\usepackage{amsfonts}
\usepackage{amssymb,latexsym}
\usepackage{enumerate}
\usepackage{mathrsfs}
\usepackage{hyperref}
 \usepackage{color}\makeatletter
\@namedef{subjclassname@2010}{%
  \textup{2010} Mathematics Subject Classification}
\makeatother

  [2002/01/19 v2.2g %
    AMS font definitions%
  ]
\DeclareFontFamily{U}{euf}{}
\DeclareFontShape{U}{euf}{m}{n}{%
  <5><6><7><8><9>gen*eufm%
  <10><10.95><12><14.4><17.28><20.74><24.88>eufm10%
  }{}
\DeclareFontShape{U}{euf}{b}{n}{%
  <5><6><7><8><9>gen*eufb%
  <10><10.95><12><14.4><17.28><20.74><24.88>eufb10%
  }{}

  [2002/01/19 v2.2g %
    AMS font definitions%
  ]
\DeclareFontFamily{U}{msb}{}
\DeclareFontShape{U}{msb}{m}{n}{%
  <5><6><7><8><9>gen*msbm%
  <10><10.95><12><14.4><17.28><20.74><24.88>msbm10%
  }{}

  [2002/01/19 v2.2g %
    AMS font definitions%
  ]
\DeclareFontFamily{U}{msa}{}
\DeclareFontShape{U}{msa}{m}{n}{%
  <5><6><7><8><9>gen*msam%
  <10><10.95><12><14.4><17.28><20.74><24.88>msam10%
  }{}

\newtheorem{theorem}{Theorem}[section]
\newtheorem{lemma}[theorem]{Lemma}

\newtheorem{corollary}[theorem]{Corollary}

\theoremstyle{definition}

\newtheorem{remark}[theorem]{Remark}

\numberwithin{equation}{section} 

\begin{document}

\title[]
{Identities for the Euler polynomials, $p$-adic integrals and Witt's formula}

\author{Su Hu}
\address{Department of Mathematics, South China University of Technology, Guangzhou, Guangdong 510640, China}
\email{mahusu@scut.edu.cn}

\author{Min-Soo Kim}
\address{Department of Mathematics Education, Kyungnam University, Changwon, Gyeongnam 51767, Republic of Korea}
\email{mskim@kyungnam.ac.kr}

%\thanks{*Corresponding author}

\begin{abstract}
By using Cauchy's formula, it is known that Bernoulli numbers and  Euler numbers can be represented by the  contour integrals
\begin{equation*}
\begin{aligned}
B_n&=\frac{n!}{2\pi i}\oint \frac{z}{e^z-1}\frac{d
z}{z^{n+1}},\label{condefi}\\[4pt]
E_n&=\frac{n!}{2\pi i}\oint \frac{2e^z}{e^{2z}+1}\frac{d
z}{z^{n+1}}, 
\end{aligned}
\end{equation*}
while the following Witt's formula represents  Euler polynomials through the fermionic $p$-adic integrals
 $$E_n(a)=\int_{\mathbb Z_p}(x+a)^nd\mu_{-1}(x).$$

Base on the above Witt's identity  and the binomial theorem, we prove some new identities for the Euler polynomials briefly.
In particular, some symmetry properties of Euler polynomials have been discovered, which implies many interesting identities (known or unknown), 
including the Kaneko-Momiyama type
identities (shown by Wu, Sun, and Pan)
and the Alzer-Kwong type identity for Euler polynomials.
\end{abstract}


%\thanks{Received May 24, 2009}

\maketitle

% \linenumbers


\def\ord{\text{ord}_p}
\def\C{\mathbb C_p}
\def\BZ{\mathbb Z}
\def\Z{\mathbb Z_p}
\def\Q{\mathbb Q_p}
\def\wh{\widehat}
\def\ov{\overline}

\section{Introduction}
\label{Intro}
\subsection{Background} 

Denote by $\mathbb N=\{1,2,\ldots\}$ and $\mathbb N_0=\mathbb N\cup\{0\}.$

The Bernoulli polynomials $B_{n}(a)$ are defined by the generating function 
\begin{equation}\label{BerP}
\frac{te^{at}}{e^t-1}=\sum_{n=0}^\infty B_n(a)\frac{t^n}{n!}
\end{equation}
and $B_{n}=B_{n}(0)$ are named Bernoulli numbers.
These numbers and polynomials arise from Bernoulli's calculations of power sums in 1713, that is, 
$$
\sum_{j=1}^{m}j^{n}=\frac{B_{n+1}(m+1)-B_{n+1}}{n+1}
$$ 
(see \cite[p. 5, (2.2)]{Sun}).  

The  Euler polynomials $E_n(a)$ are defined by the generating function
\begin{equation}\label{Eu-pol}
\frac{2e^{at}}{e^t+1}=\sum_{n=0}^\infty E_n(a)\frac{t^n}{n!}.
\end{equation}
These polynomials were introduced by Euler who studied the alternating power sums, that is,
\begin{equation*}%\label{EurP2}
\sum_{j=1}^{m}(-1)^{j}j^{n}=\frac{(-1)^{{m}}E_{n}(m+1)+E_{n}(0)}{2}
\end{equation*} 
(see \cite[p. 5, (2.3)]{Sun}). 
It is easy to see that the first few  Euler polynomials are  
\begin{equation}\label{few-va}
E_0(a)=1,~E_1(a)=a-\frac12,~E_2(a)=a^2-a,~E_3(a)=a^3-\frac32a^2+\frac14.
\end{equation}
The Euler numbers $E_n$ are defined by 
\begin{equation}\label{o-Eu-num}
E_n=2^nE_n\left(\frac12\right)
\end{equation}
and $E_{2n+1}=0$ for $n\in\mathbb N$ 
(see for instance \cite[p.~804, 23.1.2]{AS} and \cite[(0.1)]{Ma16}). The Euler numbers and polynomials (so called by Scherk in 1825) appear in Euler's famous book,
Insitutiones Calculi Differentials (1755, pp. 487--491 and p. 522) and have been found many applications 
in combinatorics, number theory, classical and $p$-adic analysis. For example, see
Abramowitz and Stegun \cite{AS},
Agoh and Dilcher \cite{AD}, Chen and Sun \cite{CS},
Gessel \cite{Ge}, He and Zhang \cite{HZ}, Kim and Hu \cite{KS}, Koblitz \cite{Ko}, Robert \cite{Ro}, Simsek \cite{Sim} and Sun \cite{Su}.

From the definition (\ref{Eu-pol}) we can easily deduce the following identities:
\begin{equation}\label{Eu-pol-3}
E_n(1-a)=(-1)^nE_{n}(a),
\end{equation}
\begin{equation}\label{Eu-pol-4}
(-1)^nE_n(-a)+E_{n}(a)=2a^n
\end{equation}
(see for instance \cite[p.~804, 23.1.8 and 23.1.9]{AS} ).

Letting $a=1$ in (\ref{Eu-pol-3}), we have a relation  between  the special values of Euler polynomials at $1$ and $0:$
\begin{equation}\label{Eu-p-sp}
E_n(1)=(-1)^nE_n(0)=\begin{cases}1 &\text{for }n=0 \\ -E_n(0) &\text{for }n\neq0,\end{cases}
\end{equation}
because $E_n(0)=0$ if $n$ is even.

In 2004, Wu, Sun and Pan \cite[(7) and (9)]{WSP} proved the following interesting formulae:
\begin{equation}\label{WSP(7)}
(-1)^{m}\sum_{i=0}^{m}\binom{m}i E_{n+i}(a)=
(-1)^{n}\sum_{j=0}^{n}\binom{n}j E_{m+j}(-a)
\end{equation}
and
\begin{equation}\label{WSP(9)}
\begin{aligned}
(-1)^{m}\sum_{i=0}^{m}\binom{m+1}i & (n+i+1)E_{n+i}(a) \\
&\quad+(-1)^{n}\sum_{j=0}^{n}\binom{n+1}j (m+j+1)E_{m+j}(-a) \\
&=(-1)^{m+1}2(m+n+2)(E_{m+n+1}(a)-a^{m+n+1}).
\end{aligned}
\end{equation}
Here $\binom mi$ denotes the usual binomial coefficient.
The identity (\ref{WSP(7)}) can be viewed as a version of the Kaneko-Momiyama type
identity for Euler polynomials.
In Section \ref{proofs}, we shall give an alternative proof of  (\ref{WSP(7)}) (see Remark \ref{rem-sun} below).
The identity (\ref{WSP(9)}) is an Euler polynomial version of Kaneko-Momiyama relations among
Bernoulli numbers.
See Kaneko \cite{Ka}, Momiyama \cite{Mom}, Gessel \cite{Ge} and Wu-Sun-Pan \cite{WSP} for details.
In fact, we can see that Theorem \ref{thm1} reduces to (\ref{WSP(9)}) by setting $q=k=1$ (see Remark \ref{rem} below).
More recently, by using Zeilberger’s algorithm \cite{Zeilberger}, Chen and Sun \cite{CS} gave  new proofs of  many existing recurrence relations between
Bernoulli polynomials.

The aim of this paper is to present new identities for the Euler polynomials which extend (\ref{WSP(7)}) and (\ref{WSP(9)}).
Our main tool is the  fermionic  $p$-adic integral on $\mathbb Z_p,$
and
in the next section we shall give a brief recall of  the definition and identities for this integral.
Using them the detail proofs for the main theorems will be shown in the final section.

\subsection{Main results} Our main results and their corollaries are as follows.
\begin{theorem}\label{thm1}
Let $k,q\in\mathbb N$ and let $m,n\in\mathbb N_0$ such that $m+n>0.$ Then for given odd integer $k,$ we have
$$\begin{aligned}
(-1)^{m}\sum_{i=0}^{m+q}\binom{m+q}i &\binom{n+q+i}k E_{n+q+i-k}(a) \\
&+(-1)^{n}\sum_{j=0}^{n+q}\binom{n+q}j\binom{m+q+j}k E_{m+q+j-k}(-a)=0.
\end{aligned}$$
\end{theorem}

\begin{remark}\label{rem}
Considering the identity (\ref{Eu-pol-4}), it is easily checked that
\begin{equation}\label{rem-9}
\begin{aligned}
(-1)^m&(m+n+2)E_{m+n+1}(a)+(-1)^n(m+n+2)E_{m+n+1}(-a) \\
&=(-1)^m(m+n+2)(E_{m+n+1}(a)+(-1)^{m+n}E_{m+n+1}(-a)) \\
&=(-1)^m2(m+n+2)(E_{m+n+1}(a)-a^{m+n+1}).
\end{aligned}
\end{equation}
Thus letting $q=k=1$ in Theorem \ref{thm1}, we recover Wu, Sun and Pan's  formula (\ref{WSP(9)}).
\end{remark}

Put  $m=n,q=k$ and $a=0$ in Theorem \ref{thm1}, we get an Euler polynomial version of
Z\'ekiri and Bencherif’s formula which was stated for Bernoulli numbers (see {\cite[Theorem 1.1]{ZB}}).
\begin{corollary}\label{cro0}
Let $n\in\mathbb N_0.$ Then for given odd integers $q,$ we have
$$\sum_{i=0}^{m+q}\binom{m+q}i (n+q+i)(n+q+i-1)\cdots(n+i+1) E_{n+i}(0)=0.$$
\end{corollary}

Note that the above identity is also an Euler polynomial version of  Kaneko's identity \cite{Ka} when $q=1,$ Chen and Sun's identity \cite{CS} when $q=3.$

Put $a=0$ and $q=k=1$ in Theorem \ref{thm1}, we have the following analogue of Momiyama's formula \cite{Mom}.

\begin{corollary}\label{cro1}
Let $m,n\in\mathbb N_0$ such that $m+n>0.$ Then we have
$$\begin{aligned}
\sum_{i=0}^{m+1}\binom{m+1}i &(n+i+1)E_{n+i}(0) \\
&+(-1)^{m+n}\sum_{j=0}^{n+1}\binom{n+1}j (m+j+1)E_{m+j}(0)=0.
\end{aligned}$$
\end{corollary}

If we put $m=n$ in Corollary \ref{cro1}, then we get the following dentity.

\begin{corollary}\label{cro2}
For any $n\in\mathbb N_0,$
$$\sum_{j=0}^{n+1}\binom{n+1}j (n+j+1)E_{n+j}(0)=0.$$
\end{corollary}

This can be regarded as an Euler polynomial version of Kaneko's formulae which was stated for
Bernoulli numbers (see \cite{Ka}). From the above identity, we have a recurrence relationship 
\begin{equation}\label{Eu-pol-re}
E_{2n+1}(0)=-\frac{1}{2(n+1)}\sum_{j=0}^{n}\binom{n+1}j (n+j+1)E_{n+j}(0),
\end{equation}
It is an interest, sinece from this, in order to to calculate 
$E_{2n+1}(0),$ we only needs half of the previous terms ($E_k(0)$ with $n\leq k\leq2n$ with $k$ odd). 

We show that the $p$-adic integral is also work for proving the following Sun's identity for Euler polynomials.

\begin{theorem}[{Sun's identity, see \cite[Theorem 1.2(iii)]{Su}}]\label{sun-thm}
We have
$$\begin{aligned}
(-1)^m\sum_{i=0}^{m}\binom{m}i a^{m-i}E_{n+i}(b)=(-1)^{n}\sum_{j=0}^{n}\binom{n}j a^{n-j}E_{m+j}(c),
\end{aligned}$$
where $a+b+c=1.$
\end{theorem}

\begin{remark}\label{rem-sun}
In the case $a=1,b=t$ and $c=-t,$ Theorem \ref{sun-thm} reduces to  (\ref{WSP(7)}).
\end{remark}

Putting $a=1,$ $b=1/2$ and $c=-1/2$ in Theorem \ref{sun-thm}, we get the following consequence which can also be found in Chen and Sun's work (see \cite[Theorem 6.1]{CS}).

\begin{corollary}\label{sun-thm-cor}
We have
$$(-1)^m\sum_{i=0}^{m}\binom{m}i \frac{E_{n+i}}{2^{n+i}}=(-1)^n\sum_{j=0}^{n}\binom{n}j E_{m+j}\left(-\frac12\right),$$
where $m,n\in\mathbb N_0.$
\end{corollary}

Motivated by the above works, we have  the following general result.

\begin{theorem}\label{thm2}
Let $s\in\mathbb N$ and let $k,m,n\in\mathbb N_0$ such that $m+n>0.$ Then we have
$$\begin{aligned}
&\delta_{s,k}\biggl( \sum_{i=0}^{m+1}(s+1)^{m-i+1}\binom{m+1}i\binom{n+i+1}{k}E_{n+i-k+1}(a)  \\
&\quad\quad +(-1)^{m+n} \sum_{j=0}^{n+1}(s+1)^{n-j+1}\binom{n+1}j\binom{m+j+1}{k}E_{m+j-k+1}(-a)
 \biggl) \\
&\quad\quad=\frac2{k!}\sum_{l=1}^s(-1)^lP_{m,n,s}^{(k)}(l;a),
\end{aligned}$$
where $P_{m,n,s}^{(k)}(l;a)$ satisfy the relation
$$\begin{aligned}
P_{m,n,s}^{(k)}(l;a)&=\frac{d^k}{d x^k}\biggl((x+a)^{m+1}(x+a-s-1)^{n+1} \\
&\qquad\qquad+(-1)^{m+n}(x-a)^{n+1}(x-a-s-1)^{m+1}\biggl)\biggl|_{x=l}
\end{aligned}$$
and $\delta_{s,k}$ can be expressed as
$$\delta_{s,k}=(-1)^s-(-1)^k=\begin{cases}
+2 &\text{for $s$ even, $k$ odd}, \\
-2 &\text{for $s$ odd, $k$ even}, \\
0 &\text{otherwise}.\\
\end{cases}$$
\end{theorem}

Using Theorem \ref{thm2}, we may obtain many (well-known or new) properties for Euler numbers and polynomials. Here are some few examples.

Put $m=n,s=1$ and $a=0$ in Theorem \ref{thm2}, we have

\begin{corollary}\label{thm2-cro1}
\begin{enumerate}
\item If $k$ is even, then
$$\sum_{i=0}^{n+1}\frac{(-1)^i}{2^i}\binom{n+1}i \binom{n+i+1}k\left((-1)^iE_{n+i-k+1}(0)+(-1)^{n}\right)=0.$$
\item If $k$ is odd, then
$$\sum_{i=0}^{n+1}\frac{(-1)^i}{2^{i}}\binom{n+1}i\binom{n+i+1}k=0.$$
\end{enumerate}
\end{corollary}

Put $m=n,s=2$ and $a=0$ in Theorem \ref{thm2}, we obtain

\begin{corollary}\label{thm2-cro2}
\begin{enumerate}
\item If $k$ is odd, then
$$\begin{aligned}
\sum_{i=0}^{n+1}(-1)^i3^{n-i+1}&\binom{n+1}i\binom{n+i+1}k  \\
&\times((-1)^iE_{n+i-k+1}(0)+(-1)^n(2^{n+i-k+1}-1))=0.
\end{aligned}$$
\item If $k$ is even, then
$$\sum_{i=0}^{n+1}(-1)^i3^{n-i+1}\binom{n+1}i\binom{n+i+1}k(2^{n+i-k+1}-1)=0.$$
\end{enumerate}
\end{corollary}

\begin{theorem}\label{thm3}
Let $k,m\in\mathbb N_0$ with $0\leq k\leq m.$ Then we have
$$\begin{aligned}
\sum_{\substack{i=0\\ m+i~\text{even}}}^{m}\binom mi&\binom{m+i}{k}E_{m+i-k}(a) \\
&=\sum_{j=0}^m(-1)^{m+j}\binom mj\binom{m+j}{k}a^{m+j-k}.
\end{aligned}$$
\end{theorem}

This can be regarded as an Euler polynomials version of Alzer and Kwong's  formulae which was stated for
Bernoulli polynomials (see \cite[Theorem 1]{AK}).

An application of Theorem \ref{thm3} leads to

\begin{theorem}\label{thm3-1}
\begin{enumerate}
\item For $0\leq k\leq m,$ we have
$$\begin{aligned}
\sum_{\substack{i=0\\ m+i~\text{even}}}^{m}\binom mi\binom{m+i}{k}\binom{m+i-k}{m-k}E_i(0)=(-1)^m\binom{m}{k}.
\end{aligned}$$
\item For $0\leq k\leq m-1,$ we have
$$\sum_{\substack{i=0\\ m+i~\text{even}}}^{m}\binom mi\binom{m+i}{k}\binom{m+i-k}{m-k-1}E_{i+1}(0)=0.$$
\item For $0\leq l\leq m-k-1,$ we have
$$\sum_{\substack{i=0\\ m+i~\text{even}}}^{m}\binom mi\binom{m+i}{k}\binom{m+i-k}{l}E_{m+i-k-l}(0)=0.$$
\item For $1\leq j\leq m$ and $0\leq k\leq m,$ we have
$$\sum_{\substack{i=j\\ m+i~\text{even}}}^{m}\binom mi\binom{m+i}{k}\binom{m+i-k}{m+j-k}E_{i-j}(0)
=(-1)^{m+j}\binom mj\binom{m+j}{k}.$$
\end{enumerate}
\end{theorem}

\begin{remark}\label{rem2}
In the case $k=0$ and $l=1,$ Theorem \ref{thm3-1}(3) turns out to be Corollary \ref{cro2} since
$E_n(0)=0$ if $n$ is even. When $k=0$ and $l=3,$ Theorem \ref{thm3-1}(3) yield the following interesting identity:
\begin{equation}\label{rem2-1}
\sum_{i=0}^{m}\binom mi(m+i)(m+i-1)(m+i-2)E_{m+i-3}(0)=0, \quad m\geq3,
\end{equation}
since $E_m(0)=0$ if $m$ is even, and it remains valid if replaced $E_{n}(0)$  by Bernoulli numbers $B_{n}$ 
(see \cite[Theorem 7.1]{CS}).
\end{remark}

\section{The fermionic $p$-adic integrals and Witt's formula}
In the next two sections, we shall use the following notations.
\begin{equation*}
\begin{aligned}
\qquad p  ~~~&- ~\textrm{an odd rational prime number}. \\
\qquad\mathbb{Z}_p  ~~~&- ~\textrm{the ring of $p$-adic integers}. \\
\qquad\mathbb{Q}_p~~~&- ~\textrm{the field of fractions of}~\mathbb Z_p.\\
\qquad\mathbb C_p ~~~&- ~\textrm{the completion of a fixed algebraic closure}~\overline{\mathbb Q}_p~ \textrm{of}~\mathbb Q_{p}.
\end{aligned}
\end{equation*}

By using Cauchy's formula, it is known that Bernoulli numbers and Euler numbers can be represented by the  contour integrals
\begin{align}
B_n&=\frac{n!}{2\pi i}\oint \frac{z}{e^z-1}\frac{d
z}{z^{n+1}},\label{condefi}\\[4pt]
E_n&=\frac{n!}{2\pi i}\oint \frac{2e^z}{e^{2z}+1}\frac{d
z}{z^{n+1}}\label{condefieuler}
\end{align}
(see \cite[(2.1) and (2.2)]{CS}).

The following Witt's formula represents  Euler polynomials through the fermionic $p$-adic integrals. 
\begin{lemma}[Witt's formula]\label{lem2}
For any $n\in\mathbb N_0$ and $x\in\mathbb Q_p,$ we have
$$\int_{\mathbb Z_p}(x+a)^nd\mu_{-1}(x)=E_n(a).$$
Furthermore, $\int_{\mathbb Z_p}d\mu_{-1}(x)=E_0(a)=1.$
\end{lemma}
In this section, we shall give a brief overview of the definition and identities for the fermionic $p$-adic integrals, for details, we refer to \cite{KS}. 
The fermionic $p$-adic integral of a function $f:\mathbb Z_p\to\C$ is defined by
\begin{equation}\label{-q-e}
I_{-1}(f)=\int_{\mathbb Z_p}f(x)d\mu_{-1}(x)=\lim_{N\rightarrow\infty}\sum_{x=0}^{p^N-1}f(x)(-1)^x,
\end{equation}
and  this limits exists if $f$ is continuous on $\Z.$
 The fermionic $p$-adic integral (\ref{-q-e}) were independently found by Katz \cite[p.~486]{Katz} (in Katz's notation, the $\mu^{(2)}$-measure), Shiratani and
Yamamoto \cite{Shi}, Christol \cite{Gi}, Osipov \cite{Osipov}, Lang~\cite{Lang} (in Lang's notation, the $E_{1,2}$-measure), T. Kim~\cite{TK} from very different viewpoints.
Recently, the fermionic $p$-adic integral on $\mathbb Z_p$  is
used by the authors \cite{KS} to present many   properties of the $p$-adic Hurwitz-type Euler zeta functions,
including Laurent series expansions, functional equations, derivative formulas, and $p$-adic Raabe formulas,
and it has also been used by Kim \cite{2019} and Ma\"iga \cite{Ma} to give some new identities and
congruences concerning Euler numbers and polynomials.

Let $D\subset\C$ is an arbitrary subset closed under $a\to a+x$ for $x\in\Z$ and $a\in D.$
That is, $D$ could be $\C\backslash\Z,\Q\backslash\Z$ or $\Z.$

For a continuous function $f$  on $\Z,$ the fermionic $p$-adic integral
\begin{equation}\label{fer-sim}
F(a)=\int_{\mathbb Z_p}f(x+a)d\mu_{-1}(x), \quad(a\in D)
\end{equation}
is then defined and satisfies the equation
\begin{equation}\label{fer-sim}
F(a+1)+F(a)=2f(a).
\end{equation}
(See, e.g., \cite[p.~2982, Theorem 2.2(1)]{KS}).
From (\ref{fer-sim}), we have
\begin{equation}\label{fer-sim-2}
F(a+q)+F(a+q-1)=2f(a+q-1),
\end{equation}
where $a\in D$ and $q\in \mathbb N.$
Adding and subtracting this identity alternatively for $q=1,2,\ldots,n$  gives the formula
\begin{equation}\label{fer-sim-3}
(-1)^{q-1}F(a+q)+F(a)=2\sum_{i=0}^{q-1}(-1)^if(a+i),
\end{equation}
where $a\in D$ and $q\in \mathbb N.$

Note that the above identity (\ref{fer-sim-3}) reduces to  \cite[Theorem 2]{TK} by setting $a=0.$

In order to prove (\ref{WSP(7)}), Theorem \ref{thm1}, \ref{sun-thm}, \ref{thm2}, \ref{thm3} and \ref{thm3-1}, we need the following lemma which has been  obtained by 
T. Kim in \cite[Lemma 1]{TK}, and
Kim and Hu in \cite[Theorem 2.2(2)]{KS}.

\begin{lemma}\label{lem1}
Let $f$  be a continuous function on $\Z.$ We have the functional equation
\begin{equation}\label{Lem}
\begin{aligned}I_{-1}(f_-)&=I_{-1}(f_1)\\&=-I_{-1}(f)+2f(0),
\end{aligned}
\end{equation}
where $f_-(x)=f(-x),f_1(x)=f(x+1)$ for all $x\in\mathbb Z_p.$
In particular, if $f$ is an even function, then
$$I_{-1}(f)=f(0).$$
\end{lemma}

Using (\ref{fer-sim}) and Lemma \ref{lem1}, we get Witt's  formula for Euler polynomials (see Lemma \ref{lem2} above).

\section{Proofs of (\ref{WSP(7)}), Theorem \ref{thm1}, \ref{sun-thm}, \ref{thm2}, \ref{thm3} and \ref{thm3-1}}\label{proofs}

In this section, we prove six results for Euler polynomials. Many more identities can be obtained from this way easily.

\begin{proof}[Proof of (\ref{WSP(7)}).]

For $a\in D,$ we consider the polynomial function
$$f(x)=(-1)^m(x+a)^m(x+a-1)^n$$
on $\mathbb Z_p.$
Then by the binomial theorem, we have
\begin{equation}\label{eq1} f(x+1)=(-1)^m(x+a+1)^m(x+a)^n=(-1)^m\sum_{i=0}^m\binom mi(x+a)^{n+i}\end{equation}
and
\begin{equation}\label{eq2} f(-x)=(-1)^n\sum_{j=0}^n\binom nj(x-a)^{m+j}.\end{equation}
Note that $f(x)$ is continuous on $\Z,$ as a product of two such functions, so
the fermionic $p$-adic integral for the functions in the above equalities are well--defined.
By integrating both sides of (\ref{eq1}) and (\ref{eq2}) respectively, we have \begin{equation}\label{eq3} \int_{\mathbb Z_p}f(x+1)d\mu_{-1}(x)=(-1)^m\sum_{i=0}^m\binom mi\int_{\mathbb Z_p}(x+a)^{n+i}d\mu_{-1}(x),\end{equation}
and \begin{equation}\label{eq4} \int_{\mathbb Z_p}f(-x)d\mu_{-1}(x)=(-1)^n\sum_{j=0}^n\binom nj\int_{\mathbb Z_p}(x-a)^{m+j}d\mu_{-1}(x).\end{equation}
  (\ref{Lem}) in Lemma \ref{lem1} leads to \begin{equation}\label{eq5} \int_{\mathbb Z_p}f(x+1)d\mu_{-1}(x)= \int_{\mathbb Z_p}f(-x)d\mu_{-1}(x).\end{equation}
Then comparing (\ref{eq3}) (\ref{eq4}) and (\ref{eq5}), we have $$(-1)^m\sum_{i=0}^m\binom mi\int_{\mathbb Z_p}(x+a)^{n+i}d\mu_{-1}(x)=(-1)^n\sum_{j=0}^n\binom nj\int_{\mathbb Z_p}(x-a)^{m+j}d\mu_{-1}(x).$$
Applying Lemma \ref{lem2}, the Witt's formula,  we get (\ref{WSP(7)}).
\end{proof}

\begin{proof}[Proof of Theorem \ref{thm1}.]

Let $a\in D,q,k\in\mathbb N$ and let $m$ and $n$ be nonnegative integers with $m+n>0.$
If we define the polynomials function $h(x)$ on $\mathbb Z_p$ by
$$h(x)=(x+a)^{m+q}(x+a-1)^{n+q}+(-1)^{m+n}(x-a-1)^{m+q}(x-a)^{n+q},$$
then we have
$$h(x+1)=h(-x).$$
Taking $\left(\frac{d}{dx}\right)^k$ on the  both sides of the above identity, we have
$$h^{(k)}(x+1)=\frac{d^k h(x+1)}{dx^k}=\frac{d^k h(-x)}{dx^k}=(-1)^kh^{(k)}(-x),$$
which gives
$$h^{(k)}(x+1)=-h^{(k)}(-x)\quad\text{for $k$ being odd.}$$
By integrating both sides of the above equality, we get
\begin{equation}\label{le-ra}
\int_{\mathbb Z_p}h^{(k)}(-x)d\mu_{-1}(x)=-\int_{\mathbb Z_p}h^{(k)}(x+1)d\mu_{-1}(x)
\end{equation}
for $k$ being odd, since $h^{(k)}(x)$ is continuous on $\Z.$

Applying  (\ref{Lem}) in Lemma \ref{lem1} to the left hand side of (\ref{le-ra}), we have
\begin{equation}\label{le-ra-1}
\int_{\mathbb Z_p}h^{(k)}(-x)d\mu_{-1}(x)=\int_{\mathbb Z_p}h^{(k)}(x+1)d\mu_{-1}(x).
\end{equation}
Combining (\ref{le-ra}) with (\ref{le-ra-1}), we get the following formula
$$\int_{\mathbb Z_p}h^{(k)}(x+1)d\mu_{-1}(x)=-\int_{\mathbb Z_p}h^{(k)}(x+1)d\mu_{-1}(x),$$
or equivalently,
\begin{equation}\label{le-ra-2}
\int_{\mathbb Z_p}h^{(k)}(x+1)d\mu_{-1}(x)=0
\end{equation}
for $k$ being odd.
To calculate the left-hand side of the above equation, we rewrite $h(x+1)$ in the form
$$h(x+1)=\sum_{i=0}^{m+q}\binom{m+q}i(x+a)^{n+q+i}+(-1)^{m+n}\sum_{j=0}^{n+q}\binom{n+q}j(x-a)^{m+q+j}.$$
Then
$$\begin{aligned}
h^{(k)}(x+1)&=k!\sum_{i=0}^{m+q}\binom{m+q}i\binom{n+q+i}{k}(x+a)^{n+q+i-k} \\
&\quad+(-1)^{m+n}k!\sum_{j=0}^{n+q}\binom{n+q}j\binom{m+q+j}{k}(x-a)^{m+q+j-k}.
\end{aligned}$$
Applying (\ref{le-ra-2}) to the above $k$-th derivative $h^{(k)}(x+1),$ and use Lemma \ref{lem2}, the Witt's formula,  we conclude our assertion for $k$ being odd.
\end{proof}

\begin{proof}[Proof of Theorem \ref{sun-thm}.]

For $a+b+c=1,$ we consider the polynomial function
$$g(x)=(-1)^m(x+a+b-1)^m(x+b-1)^n$$
on $\mathbb Z_p.$

(1) We have
\begin{equation}\label{eq1s} g(x+1)=(-1)^m\sum_{i=0}^m\binom mi a^{m-i}(x+b)^{n+i}.\end{equation}

(2) Since $a+b+c=1,$ we have 
\begin{equation}\label{eq2s}
\begin{aligned}
g(-x)&=(-1)^m(-x+a+b-1)^m(-x+b-1)^n \\
&=(-1)^n(x-a-b+1)^m(x-b+1)^n \\
&=(-1)^n(x+c)^m(x+a+c)^n \\
&=(-1)^n\sum_{j=0}^n\binom nja^{n-j}(x+c)^{m+j}.
\end{aligned}
\end{equation}
By integrating both sides of (\ref{eq1s}) and (\ref{eq2s}) respectively, we have \begin{equation}\label{eq3s} \int_{\mathbb Z_p}g(x+1)d\mu_{-1}(x)=(-1)^m\sum_{i=0}^m\binom mi a^{m-i}\int_{\mathbb Z_p}(x+b)^{n+i}d\mu_{-1}(x)\end{equation}
and \begin{equation}\label{eq4s} \int_{\mathbb Z_p}g(-x)d\mu_{-1}(x)=(-1)^n\sum_{j=0}^n\binom nj a^{n-j}\int_{\mathbb Z_p}(x+c)^{m+j}d\mu_{-1}(x).\end{equation}
(\ref{Lem}) in Lemma \ref{lem1} leads to \begin{equation}\label{eq5s} \int_{\mathbb Z_p}g(x+1)d\mu_{-1}(x)= \int_{\mathbb Z_p}g(-x)d\mu_{-1}(x).\end{equation}
Then comparing (\ref{eq3s}) (\ref{eq4s}) and (\ref{eq5s}), 
we get
$$\begin{aligned}
(-1)^m\sum_{i=0}^m\binom mi a^{m-i}&\int_{\mathbb Z_p}(x+b)^{n+i}d\mu_{-1}(x) \\
&=(-1)^n\sum_{j=0}^n\binom nj a^{n-j}\int_{\mathbb Z_p}(x+c)^{m+j}d\mu_{-1}(x).
\end{aligned}$$
This implies Theorem \ref{sun-thm} by using Lemma \ref{lem2}.
\end{proof}

\begin{proof}[Proof of Theorem \ref{thm2}.]

Let $a\in D,s\in\mathbb N$ and let $k,m,n\in\mathbb N_0$ such that $m+n>0.$
Let us define the polynomials $P_{m,n,s}(x;a)$ by
\begin{equation}\label{pf-1}
\begin{aligned}
P_{m,n,s}(x;a)&=(x+a)^{m+1}(x+a-s-1)^{n+1} \\
&\quad+(-1)^{m+n}(x-a)^{n+1}(x-a-s-1)^{m+1}
\end{aligned}
\end{equation}
on $\mathbb Z_p.$
Then we have
$$P_{m,n,s}(x+s+1;a)=P_{m,n,s}(-x;a),$$
which gives
\begin{equation}\label{k-diff}
P_{m,n,s}^{(k)}(x+s+1;a)=(-1)^kP_{m,n,s}^{(k)}(-x;a).
\end{equation}
It is easily seen that
$$\begin{aligned}
\sum_{l=1}^s(-1)^l&\left(P_{m,n,s}^{(k)}(x+j+1;a)+P_{m,n,s}^{(k)}(x+j;a)\right) \\
&=-P_{m,n,s}^{(k)}(x+1;a)+(-1)^sP_{m,n,s}^{(k)}(x+s+1;a).
\end{aligned}$$
Therefore we have
$$\begin{aligned}
(-1)^s&\int_{\mathbb Z_p}P_{m,n,s}^{(k)}(x+s+1;a)d\mu_{-1}(x) \\
&=\sum_{l=1}^s(-1)^l\int_{\mathbb Z_p}\left(P_{m,n,s}^{(k)}(x+l+1;a)+P_{m,n,s}^{(k)}(x+l;a)\right)d\mu_{-1}(x) \\
&\quad+\int_{\mathbb Z_p}P_{m,n,s}^{(k)}(x+1;a)d\mu_{-1}(x) \\
&=2\sum_{l=1}^s(-1)^lP_{m,n,s}^{(k)}(l;a)+\int_{\mathbb Z_p}P_{m,n,s}^{(k)}(-x;a)d\mu_{-1}(x) \\
&\quad(\text{using (\ref{fer-sim}) and (\ref{Lem}) in Lemma \ref{lem1}}) \\
&=2\sum_{l=1}^s(-1)^lP_{m,n,s}^{(k)}(l;a)+(-1)^k\int_{\mathbb Z_p}P_{m,n,s}^{(k)}(x+s+1;a)d\mu_{-1}(x) \\
&\quad(\text{using (\ref{k-diff})}) \\
\end{aligned}$$
yields that
\begin{equation}\label{pf-2}
\delta_{s,k}\int_{\mathbb Z_p}P_{m,n,s}^{(k)}(x+s+1;a)d\mu_{-1}(x)=2\sum_{l=1}^s(-1)^lP_{m,n,s}^{(k)}(l;a),
\end{equation}
where $\delta_{s,k}=((-1)^s-(-1)^k).$
From (\ref{pf-1}) and the binomial theorem, we have
$$\begin{aligned}
P_{m,n,s}(x+s+1;a)&=\sum_{i=0}^{m+1}\binom{m+1}i(s+1)^{m-i+1}(x+a)^{n+i+1} \\
&\quad+(-1)^{m+n}\sum_{j=0}^{n+1}\binom{n+1}j(s+1)^{n-j+1}(x-a)^{m+j+1},
\end{aligned}$$
which implies 
\begin{equation}\label{pf-3}
\begin{aligned}
P_{m,n,s}^{(k)}(x+s+1;a)&=k!\sum_{i=0}^{m+1}\binom{m+1}i\binom{n+i+1}{k}(s+1)^{m-i+1}\\
&\qquad\times(x+a)^{n+i-k+1} \\
&\quad+(-1)^{m+n}k!\sum_{j=0}^{n+1}\binom{n+1}j\binom{m+j+1}{k}(s+1)^{n-j+1}\\
&\qquad\times(x-a)^{m+j-k+1}.
\end{aligned}
\end{equation}
Since  $\int_{\mathbb Z_p}d\mu_{-1}(x)=E_0(a)=1,$
substituting (\ref{pf-3}) into the left hand side of (\ref{pf-2}), then using Lemma \ref{lem2}, we get
$$\begin{aligned}
&\delta_{s,k}\biggl( \sum_{i=0}^{m+1}(s+1)^{m-i+1}\binom{m+1}i\binom{n+i+1}{k}E_{n+i-k+1}(a)  \\
&\quad\quad +(-1)^{m+n} \sum_{j=0}^{n+1}(s+1)^{n-j+1}\binom{n+1}j\binom{m+j+1}{k}E_{m+j-k+1}(-a)
 \biggl) \\
&\quad\quad=\frac2{k!}\sum_{l=1}^s(-1)^lP_{m,n,s}^{(k)}(l;a).
\end{aligned}$$
This completes our proof.
\end{proof}

\begin{proof}[Proof of Theorem \ref{thm3}.]

For $a\in D,$ we consider the polynomial function
$$q(x)=(x+a)^m(x+a-1)^m$$ 
on $\mathbb Z_p.$
By the binomial theorem, the formula $q(x)$ and $q(x+1)$ can be rewritten as
$$q(x)=\sum_{i=0}^m(-1)^{m+i}\binom mi(x+a)^{m+i},$$
$$q(x+1)=\sum_{i=0}^m\binom mi(x+a)^{m+i},$$
respectively.
Hence, we have
\begin{equation}\label{3-1}
\begin{aligned}
q^{(k)}(x+1)+q^{(k)}(x)
&=k!\sum_{i=0}^m\binom mi\binom{m+i}k(x+a)^{m+i-k}\left(1+(-1)^{m+i} \right) \\
&=\begin{cases}
2k!\sum_{i=0}^m\binom mi\binom{m+i}k(x+a)^{m+i-k} &\text{if $m+i$ even} \\
0 &\text{if $m+i$ odd}
\end{cases}
\end{aligned}
\end{equation}
and
\begin{equation}\label{3-2}
q^{(k)}(0)=k!\sum_{j=0}^m(-1)^{m+j}\binom mj\binom{m+j}ka^{m+j-k}.
\end{equation}
The second equality in (\ref{Lem}) of Lemma \ref{lem1} implies
\begin{equation}\label{3-3}
\int_{\mathbb Z_p}(q^{(k)}(x+1)+q^{(k)}(x))d\mu_{-1}(x)=2q^{(k)}(0).
\end{equation}
On expanding (\ref{3-3}) by (\ref{3-1}) and (\ref{3-2}), we obtain
$$\begin{aligned}
\sum_{\substack{i=0\\ m+i~\text{even}}}^{m}\binom mi&\binom{m+i}{k}\int_{\mathbb Z_p}(x+a)^{m+i-k} d\mu_{-1}(x) \\
&=\sum_{j=0}^m(-1)^{m+j}\binom mj\binom{m+j}{k}a^{m+j-k},
\end{aligned}$$
and the result follows from Lemma \ref{lem2}.
\end{proof}

\begin{proof}[Proof of Theorem \ref{thm3-1}.]

We consider the polynomial function
$$r(x)=x^m(x-1)^m$$ 
on $\mathbb Z_p.$
By the binomial theorem, the formula $r(x)$ and $r(x+1)$ can be rewritten as
$$r(x)=\sum_{i=0}^m(-1)^{m+i}\binom mix^{m+i}\quad\text{and}\quad r(x+1)=\sum_{i=0}^m\binom mix^{m+i},$$
respectively.
Hence, we have
\begin{equation}\label{3-1-r}
\begin{aligned}
r^{(k)}(x+1)+r^{(k)}(x)
&=k!\sum_{i=0}^m\binom mi\binom{m+i}kx^{m+i-k}\left(1+(-1)^{m+i} \right) \\
&=\begin{cases}
2k!\sum_{i=0}^m\binom mi\binom{m+i}kx^{m+i-k} &\text{if $m+i$ even} \\
0 &\text{if $m+i$ odd}
\end{cases}
\end{aligned}
\end{equation}
and
\begin{equation}\label{3-2-r}
r^{(k)}(x)=k!\sum_{j=0}^m(-1)^{m+j}\binom mj\binom{m+j}kx^{m+j-k}.
\end{equation}

To see Part (1), note that by (\ref{3-1-r}),
\begin{equation}\label{3-1-1}
\begin{aligned}
\left(\frac{d}{dx}\right)^{m-k}&\biggl(r^{(k)}(x+1)+r^{(k)}(x)\biggl) \\
&=2k!(m-k)!\sum_{\substack{i=0\\ m+i~\text{even}}}^{m}\binom mi\binom{m+i}k\binom{m+i-k}{m-k}x^{i}.
\end{aligned}
\end{equation}
Similarly, by (\ref{3-2-r}) we obtain
\begin{equation}\label{3-1-2}
\begin{aligned}
\left(\frac{d}{dx}\right)^{m-k}&\biggl(r^{(k)}(x)\biggl)\biggl|_{x=0} \\
&=k!(m-k)!\sum_{j=0}^{m}(-1)^{m+j}\binom mj\binom{m+j}k\binom{m+j-k}{m-k}x^{j}\biggl|_{x=0} \\
&=k!(m-k)!(-1)^m\binom mk,
\end{aligned}
\end{equation}
since $0^j=1$ if $j=0$ and $0^j=0$ if $j\in\mathbb N.$

The second equality in (\ref{Lem}) of Lemma \ref{lem1} implies
\begin{equation}\label{3-3-2}
\int_{\mathbb Z_p}(r^{(m)}(x+1)+r^{(m)}(x))d\mu_{-1}(x)=2r^{(m)}(0).
\end{equation}
Substituting (\ref{3-1-1}) and (\ref{3-1-2}) into (\ref{3-3-2}), we have
\begin{equation}\label{3-1-3}
\sum_{\substack{i=0\\ m+i~\text{even}}}^{m}\binom mi\binom{m+i}{k}\binom{m+i-k}{m-k}
\int_{\mathbb Z_p}x^id\mu_{-1}(x)
=(-1)^m\binom{m}{k},
\end{equation}
which leads to Part (1) by using Lemma \ref{lem2}, with $a=0.$

%The proofs of (2), (3) and (4) follow exactly along the lines of the proof of (1).
The Parts (2), (3) and (4) can be shown in a similar way with Part (1).
\end{proof}

\bibliography{central}

\end{document}